\documentclass[11pt]{article}

\usepackage{amsthm}
\usepackage{amsmath}
\usepackage{thmtools,thm-restate}
\usepackage{amssymb}
\usepackage{graphicx}
\usepackage[margin=2cm, includefoot]{geometry}
\usepackage{commath} 
\usepackage{mathtools}
\usepackage{xcolor}
\usepackage{color}
\usepackage{url}
\usepackage{algorithm}
\usepackage{bm}
\usepackage{verbatim}
\usepackage{subcaption}
\usepackage{enumitem}
\usepackage{algorithmic}
\usepackage{comment}
\usepackage{titlecaps}
\usepackage{comment}

\usepackage{hyperref}   
\hypersetup{colorlinks=true, linkcolor=blue, filecolor=magenta, urlcolor=cyan, pdftitle={Hermite schemes over manifolds}, pdfauthor=author}

\usepackage{authblk}

\newcommand{\R}{\mathbb{R}}
\newcommand{\N}{\mathbb{N}}

\newcommand{\Z}{\mathbb{Z}}

\numberwithin{equation}{section}

\newtheorem{thm}{Theorem}
\numberwithin{thm}{section} 

\newtheorem{cor}[thm]{Corollary}

\newtheorem{lemma}[thm]{Lemma}

\newtheorem{remark}[thm]{Remark}

\begin{document}


\title{Improving the convergence analysis of linear subdivision schemes}

\author{Nira Dyn}
\author{Nir Sharon\thanks{corresponding author}}

\affil{%
    School of Mathematical Sciences, Tel Aviv University, Tel Aviv, Israel}

\date{\vspace{-5ex}}

\maketitle
\begin{abstract}
This work presents several new results concerning the analysis of the convergence of binary, univariate, and linear subdivision schemes, all related to the 
{\it contractivity factor} of a convergent scheme. First, we prove that a convergent scheme cannot have a contractivity factor lower than half. Since the lower this factor is, the faster is the convergence of the scheme, schemes with contractivity factor $\frac{1}{2}$, such as those generating spline functions, have optimal convergence rate.

Additionally, we provide further insights and conditions for the convergence of linear schemes and demonstrate their applicability in an improved algorithm for determining the convergence of such subdivision schemes.
\end{abstract}

\section{Introduction}

This paper presents an improvement of the known analysis concerning the convergence of binary, univariate, linear subdivision schemes. Such a scheme repeatedly refines sequences of numbers according to simple linear refinement rules. If the subdivision scheme is convergent, namely, it uniformly converges to a limit from any initial data sequence, then, it generates in the limit a graph of a univariate continuous limit function. These schemes generate curves in $\R^n$  from initial data consisting of a sequence of vectors (points) in $\R^n$ by operating on each component separately. Given refinement rules, which are defined by a finite sequence of numbers (see Section 2.1), it is necessary to determine if the scheme is convergent and how fast. We address these two questions in Section 3 and improve the algorithm to determine if a scheme is convergent. Before, in Section 2, we give a short introduction to binary univariate linear subdivision schemes and their convergence. We chose this class of schemes because the ideas leading to our improvements are simpler to present for this class. We believe that our ideas can be extended to a rather wide class of binary, multivariate, linear subdivision schemes.

\section{A short introduction to univariate linear subdivision schemes} \label{subsec:subdivision}
This section introduces notions, notation, and techniques related to the convergence analysis of binary, univariate, and linear subdivision schemes. 

\subsection{Convergence}
Two refinement rules of a  binary, univariate, linear subdivision scheme $S_a$ have the form 
\begin{equation} \label{eqn:refinement}
    (S_a{f})_i=\sum_{j\in \Z} {\bf a}_{i-2j}{f}_j. 
\end{equation} 
Here $f$ and $ S_a(f)$ are two sequences of real numbers defined on $\mathbb Z$, and $\bf a$ is a finite sequence of real numbers termed the mask of the scheme.
All sequences $f$ considered here satisfy $\norm{f}<\infty$ and by the finiteness of the mask of the scheme so are all sequences obtained by the subdivision scheme. Unless otherwise stated, the norm used here is the supremum norm, either over the real line for functions or over $\Z$ for sequences.

Note that~\eqref{eqn:refinement} consists of two refinement rules: one for the even indices $i$ and the other for the odd ones. The repeated application of the refinement rules of the scheme $S_a$ to the sequence $ f^0 $ leads to the sequence of sequences,
\[ { f}^{k+1}= S_a({f}^k), \quad k=0,1,\ldots. \] 

Subdivision schemes with refinement rules $a_0=1,\ a_{2i}=0,\ i\neq 0$ are termed 
interpolatory schemes, because $f^{k+1}_{2i}=f^k_i,\ i\in \Z$ and the limit function interpolates the points $(i,f^0_i), i\in \Z$.

The subdivision scheme $S_a$ is termed (uniformly) convergent if for any initial data $f^0=\left\{f_{i}^{0}: i \in \mathbb{Z}\right\}$, the sequence of polygonal lines $\left\{\mathbf{f}^{k}(t): k \in \mathbb{Z}_{+}\right\}$, converges uniformly, where
\[
\mathbf{f}^{k}(t) \in \Pi_{1}, \, t \in(i 2^{-k},(i+1) 2^{-k}), \quad \text{ and } \quad \mathbf{f}^{k}\left(2^{-k} i\right)=f_{i}^{k}, \quad i \in \mathbb{Z} .
\]
The limit function, which is necessarily continuous, is denoted by $S_a^{\infty} f^0$. It is also required that there exists $f^0$ such that $S_a^{\infty} f^0 \not \equiv 0$.

The first requirement in the above definition of convergence is equivalent to the existence of a continuous function $f^*$, such that for any closed interval $I \subset \mathbb{R}$
\begin{equation} \label{eqn:uniform_convergence}
    \lim _{k \rightarrow \infty} \sup _{i \in 2^{k} I}\left|f_{i}^{k}-f^*\left(2^{-k} i\right)\right|=0 .
\end{equation}
Clearly $f^*=S_a^{\infty} f^0$.

For $f=\left\{f_{i}: i \in \mathbb{Z}\right\}$ let $\Delta f=\left\{(\Delta f)_{i}=f_{i+1}-f_{i}: i \in \mathbb{Z}\right\}$. Then, the convergence of a linear subdivision scheme~\eqref{eqn:uniform_convergence} is obtained if and only if there exist $\mu \in (0,1)$ and $L\in \N$ such that
\[  \Delta S^L_a(f) \le \mu \Delta f . \]
We term the factor $\mu$ \textit{contractivity factor} and the integer $L$ \textit{contractivity number}. In such a case, the scheme is  convergent; for proof, see~\cite[Theorem 3.2]{dyn1992subdivision_book}. Furthermore, the contractivity factor and number imply that there exists a constant $C$, independent of $k$, such that $\norm{ \mathbf{f}^{(k+1)L}(t)-\mathbf{f}^{kL}(t) }\le C \mu^k$.
As a result, the convergence rate of $S_a^L$ is bounded by
\[ \norm{\mathbf{f}^{kL}-f^* } \le \sum_{j=k}^\infty \norm{\mathbf{f}^{jL}-\mathbf{f}^{(j+1)L}} \le \sum_{j=k}^\infty C \mu^j = C \mu^k \sum_{j=0}^\infty \mu^j =\left(\frac{C}{1-\mu} \right) \mu^k . \] 
In other words, the contractivity factor directly indicates the convergence rate of the subdivision scheme $S_a^L$ to its limit. 

\subsection{Notation and the analysis via Laurent polynomials}

Consider the univariate, linear subdivision scheme $S_a$ with the mask $\mathbf{a}=\left\{a_{i}: i \in \mathbb{Z}\right\}$. We define its symbol
\[
a(z)=\sum_{i \in \mathbb{Z}} a_{i} z^{i} .
\]
Since the schemes we consider have masks of finite support, the corresponding symbols are Laurent polynomials, namely polynomials in positive and negative powers of the variable z. 

The notion of Laurent polynomials enables us to write the refinement in an algebraic form. Let $F(z ; f)=\sum_{j \in \mathbb{Z}} f_{j} z^{j}$ be a formal generating function associated with the sequence $f$. Then, \eqref{eqn:refinement} becomes
\[
F\left(z ; S_{a} f\right)=a(z) F\left(z^{2} ; f\right) .
\]
It follows from the above equality that the symbol of $S^L_a$ is
\begin{equation}
\label{eq:symbol_power}
a^L(z)=a(z)a^{L-1}(z^2)=a(z)a(z^2)a(z^4) \cdot \cdot \cdot a(z^{2(L-1)})
\end{equation}

A known theorem, see~\cite{dyn1992subdivision_book} states that if $S_a$ is a convergent subdivision scheme, then 
\begin{equation} \label{eqn:necessary_cond}
  \sum_{j} a_{2 j}=\sum_{j} a_{2j+1}=1 \, .
\end{equation}
Therefore, the symbol of a convergent subdivision scheme satisfies,
\begin{equation} \label{eqn:necessary_cond_2}
    a(-1)=0, \quad a(1)=2 .
\end{equation}
Namely, the symbol factorizes into $a(z)=(1+z) q(z)$ for some Laurent polynomial $q$ with $q(1)=1$. Interestingly, the subdivision $S_{q}$, defined by the coefficients of the symbol $q(z)$ is related to $S_{a}$ with symbol $a(z)$ as follows, 
\begin{equation} \label{eqn:Sa_Sq}
\Delta\left(S_{a} f\right)=S_{q} \Delta f .
\end{equation}
Relation~\eqref{eqn:Sa_Sq} implies 
\begin{equation}
    \norm{\Delta f^{k+1}} \le  \Bigl\Vert{S_q}\Bigr\Vert \norm{\Delta f^{k} } .
\end{equation}
Note that $\norm{S_q}$ is an induced norm of $S_q$, implying that the minimal number $\eta >0$ such that for all $f$ with $\norm{f}<\infty$, $\ \norm{\Delta S_q f} \le \eta \norm{\Delta f }$ is $\eta = \norm{S_q}$.

\subsection{Determining the convergence of a given scheme}

The classical formulation using Laurent polynomials suggests an algorithm to determine the convergence of a scheme with a given mask. This algorithm seeks an indication of convergence in the form of a contractivity factor of the scheme or one of its powers since
\begin{equation} \label{eqn:contractivity_high_powers}
    \norm{\Delta f^{k}}= \norm{S_q^k(\Delta f^{0})} \le   \norm{S_q^k} \norm{\Delta f^{0}}, \quad k\in \N. 
\end{equation}
The algorithm gets the maximum power $M$ as an input parameter and computes $\norm{S_q^k}$, $ k \le M$ to find a contractivity factor. Then, if no contractivity is obtained, namely, $\norm{S_q^k}\ge 1$ for all $k \le M$, the algorithm returns that the convergence question is inconclusive. Otherwise the algorithm  determines $L$ as the first $k$ satisfying 
$\norm{S_q^k} < 1$, and $\mu=\norm{S_q^L}$.
Note that 
large $L$ means 
 operating $L$ steps of refinement in a row. However, practically, a high value $L$ means a slow convergence. In particular, for a given $L$ if $\norm{S_{q}^{L}}=\mu<1$, then
\begin{equation} \label{eqn:contractivity_with_L}
    \norm{\Delta f^{k}}  \leq 
    \norm{S_q^L(\Delta f^{0})}^{\left[\frac{k}{L}\right]} \norm{\Delta f^{k-L\left[\frac{k}{L}\right]}} \le
    \mu^{\left[\frac{k}{L}\right]} \max_{0 \leq \ell<L}\norm{\Delta f^{\ell}}, \quad k \in \N .
\end{equation}
The above equation implies that the average contractivity factor of each step of refinement is of order $\mu^{\frac{1}{L}}$, which is larger than $\mu$.
Therefore, as $L$ grows, refining by such a scheme requires more steps to reach convergence. 

An algorithm that concludes the above is suggested in~\cite{dyn2002analysis} and is given for completeness in Algorithm~\ref{alg:conver_test}.

\begin{algorithm}[ht]
\caption{Convergence of linear subdivision scheme}
\label{alg:conver_test}
\begin{algorithmic}[1]
\REQUIRE The symbol $a(z)$. The maximum number of iterations $M$.
\ENSURE The convergence status of the scheme
\IF{$a(1) \neq 2$ or $a(-1) \neq 0$} 
\RETURN \textit{``The scheme does not converge''}
\ENDIF
\STATE $q(z) \gets \frac{a(z)}{1+z}$.
\STATE Set $q_{1}(z)=q(z)$ and denote $q_{1}(z)=\sum_{i} q_{i}^{[1]} z^{i}$.
\FOR{$L=1, \ldots, M$}
\IF{$\max_{0 \leq i<2^{L}} \sum_{j\in \Z}\abs{q_{i-2^{L} j}^{[L]}}<1$} 
\RETURN \textit{``The scheme is convergent''}
\ELSE
\STATE  Set $q_{L+1}(z)=q(z) q_{L}\left(z^{2}\right)$ and denote $q_{L+1}(z)=\sum_{i} q_{i}^{[L+1]} z^{i}$.
\ENDIF
\ENDFOR     
\RETURN \textit{``Inconclusive! $\norm{S^k_{q}}\ge 1$ 
for all} $k \le M$''.
\end{algorithmic}
\end{algorithm}

\section{Improving the convergence analysis of univariate linear schemes}

In this section, we derive results that relate the existence of a contractivity factor to the mask coefficients of $q(z)$. These results are used to improve Algorithm 
\ref{alg:conver_test}.

The first result in the first Subsection is a lower bound on on the size of a contractivity factor with a simple proof. 
\subsection{Half is the smallest possible contractivity factor} 

We start with an auxiliary result:
\begin{lemma} \label{lemma:affine_sum}
   Let $a,b\in \R$ satisfy $a+b=1$. Then, $\max \{ \abs{a} ,\abs{b} \} \ge \frac{1}{2}$
\end{lemma}
\begin{proof}
    If $a\ge \frac{1}{2}$ we are done since then 
    $\max \{ \abs{a} ,\abs{b} \} \ge \abs{a} \ge \frac{1}{2}$. Otherwise, $a<\frac{1}{2}$ and so $b=1-a > 1- \frac{1}{2} = \frac{1}{2}$.
\end{proof}
The next theorem leads to the main result of this subsection. 
\begin{thm} \label{thm:biggerthanhalf}
    Let $S_a$ be a convergent scheme. Then, for all $L \in \mathbb{N}$,
    \begin{equation}
        \norm{S_q^L} \ge \frac{1}{2} . 
    \end{equation}
\end{thm}
\begin{proof}
    Since $q(z) = \frac{a(z)}{1+z}$, we have that $q(1) = 1$. In view of equation \eqref{eq:symbol_power}, the symbol of $S_q^L$ is  
    \begin{equation} \label{eqn:symbol_qL}
    q^L(z) = q(z)q(z^2)\cdots q(z^{{2^{L-1}}}).
    \end{equation}
    Therefore, we have that $q^L(1) = 1$ and $q^L(-1)=q(-1)$. Denote by 
    \begin{equation} \label{eqn:SeL_SoL}
      S_e^L = \sum_{j}q^L_{2j} \quad \text{  and   } \quad S_o^L = \sum_{j}q^L_{2j+1}.  
    \end{equation}
     Then, $S^L_e+S^L_o=1$. We use the definition of $\norm{S_q^L}$ and the triangle inequality to derive by Lemma~\ref{lemma:affine_sum} that, 
    \begin{equation} \label{eqn:bounding_sqLnorm}
        \norm{S_q^L}  = \max \biggl\{ \sum_{j}\abs{q^L_{2j+1}},\sum_{j}\abs{q^L_{2j}} \biggr\} \ge \max \biggl\{ \abs{S^L_o}, \abs{S^L_e} \biggr\} \ge \frac{1}{2}
    \end{equation}
    As required. 
\end{proof}
Theorem~\ref{thm:biggerthanhalf} together with~\eqref{eqn:contractivity_with_L} means the following: 
\begin{cor}
    The contractivity factor of a convergent univariate linear subdivision scheme is not smaller than half.
\end{cor}
\begin{remark}
    There exist linear subdivision schemes that satisfy the optimal contractivity factor of $\frac{1}{2}$. For example, subdivision schemes that generate splines of degree $m\ge 1$ have the symbol $ a^{[m]}(z) = \frac{(1+z)^{m+1}}{2^{m}}$. Therefore, $q(z) = \left( \frac{1+z}{2}\right)^m$ and so $\norm{S_q} = \frac{1}{2}$, that is an optimal contractivity with $L=1$.
\end{remark}

\subsection{Determining convergence}

The notions $S_e$ and $S_o$ as defined in~\eqref{eqn:SeL_SoL} play a central role in the following analysis. By the necessary condition~\eqref{eqn:necessary_cond_2} $S_q$ exists and satisfies $q(1)=1$. By~\eqref{eqn:symbol_qL} we have $q^L(1) = 1$ and $q^L(-1) = q(-1)$, and by~\eqref{eqn:SeL_SoL} we get that $S_e^L \pm S_o^L = q(\pm 1)$. Therefore,
\begin{equation} \label{eqn:se_so_indepndentL}
    S_e^L=\frac{1+q(-1)}{2}\quad \text{  and  }  \quad S_o^L=\frac{1-q(-1)}{2} .
\end{equation} 
In other words,
\begin{lemma}
   For all $L\in \N$, $S_e^L=S_e$ and $S_o^L=S_o$.
\end{lemma}
 One conclusion from the above is that the case $S_e=S_o =\frac{1}{2}$ implies that $q(-1) = 0$. Note that the multiplicity of $z+1$ in the symbol $a(z)$ is crucial for smoothness analysis~\cite{dyn2002analysis}. Moreover, it is possible to prove the following:
\begin{thm} \label{thm:positive_coefs}
    Let $S_a$ be a linear scheme that satisfies~\eqref{eqn:necessary_cond}. If the coefficients of $q(z)$ are non-negative, then 
    $S_a$ is convergent, with contractivity factor $\mu=\max\{S_e, S_o\}<1$.
\end{thm}
\begin{proof}
  Since $q(1)=1$, we have that $S_e + S_o=1$. Therefore, the positivity guarantees a contractivity factor:
  \[ \norm{S_q^L}  = \max\biggl\{ \sum_{j}\abs{q^L_{2j+1}},\sum_{j}\abs{q^L_{2j}} \biggr\} = \max\{ S_e, S_o \} < 1\]
\end{proof}

Note that non-negative coefficients of the difference scheme $S_q$ implies  non-negative coefficients of $S_a$, since $a(z)=(1+z)q(z)$. Moreover, if the zero coefficients of $q(z)$ inside its support are isolated, then the coefficients of $a(z)$ are positive. For such schemes, the convergence is also guaranteed by~\cite{melkman1996subdivision}, except for interpolatory schemes.

The next result gives a simple criterion to check if a scheme is not convergent.
\begin{thm} \label{thm:diverge_by_q}
    Let $S_a$ satisfy~\eqref{eqn:necessary_cond}. If $\abs{q(-1)}>1$, then $S_a$ does not converge.
\end{thm}
\begin{proof}
 The condition $\abs{q(-1)}>1$ implies by~\eqref{eqn:se_so_indepndentL} that one of $S_e, S_o$ is larger than $1$ and the other is negative. Therefore, since $S_e^L=S_e$ and $S_o^L=S_o$ are independent of $L$,
  \[ \norm{S_q^L}  = \max\biggl\{ \sum_{j}\abs{q^L_{2j+1}},\sum_{j}\abs{q^L_{2j}} \biggr\} \ge \max\{ \abs{S_e}, \abs{S_o} \} > 1 . \]
\end{proof}
\begin{remark}
    Theorem~\ref{thm:diverge_by_q} indicates that once we evaluate $q(-1)$ and obtain $\abs{q(-1)}>1$, there is no need to verify whether $ \norm{S_q^L}<1$ for $L>1$. In general, by~\eqref{eqn:se_so_indepndentL}, if  $\norm{S_q}$ can be determined by $S_e$ and $S_o$, then there is no point in taking powers of $ S_q$ to check contractivity.
\end{remark}

\begin{cor} \label{cor:se_so_half}
    Let $S_a$ be a convergent scheme. Then, $S_e,S_o \in [0,1]$. If, in addition, $z=-1$ is a root of multiplicity  at least two of $a(z)$, then,
    \begin{equation}
        S_e=S_o=\frac{1}{2} .
    \end{equation}
\end{cor}
\begin{proof}
    Since the scheme is convergent, Theorem~\ref{thm:diverge_by_q}implieas that $\abs{q(-1)}\le 1$ Therefore, by~\eqref{eqn:se_so_indepndentL}, we get that $S_e,S_o \in [0,1]$. If $z=-1$ is a root of  multiplicity  at least two of $a(z)$, we have the special case $q(-1)=0$ 
    and $S_e=S_o=\frac{1}{2}$.
\end{proof}
The above case where both $S_e$ and $S_o$ are in $[0,1]$ is inconclusive in terms of determining convergence, since a negative coefficient may result in large $\norm{s_q}$ 
If the coefficients are non-negative, following Theorem~\ref{thm:positive_coefs} and Corollary~\ref{cor:se_so_half}, we can conclude:
\begin{cor}
    Let $S_a$ be a scheme such that the coefficients of $q(z)$ are non-negative. If $z=-1$ is a root of multiplicity two at least of $a(z)$, then the scheme has the optimal contractivity factor $\frac{1}{2}$.
\end{cor}

We anticipate that the results of this section can be extended to multivariate schemes, over $\R^d$, with symbols of the form $$a(z_1,z_2,\ldots,z_d) = (1+z_1)(1+z_2)\cdots (1+z_d)q(z). $$

\subsection{Application to the algorithm for determining convergence}

We 
present our improved version of Algorithm~\ref{alg:conver_test} in Algorithm~\ref{alg:conver_test_improved}, incorporating the results of Theorem~\ref{thm:positive_coefs} and Theorem~\ref{thm:diverge_by_q}.

As a final remark regarding the applicability of Algorithm~\ref{alg:conver_test_improved}, we consider the question of determining the smoothness of the limits of a given subdivision scheme. 
Relation~\eqref{eqn:Sa_Sq} implies that 
$$\frac{\Delta f^{k+1}}{2^{-(k+1)}} = 2 S_q \left(\frac{\Delta f^k}{2^{-k}}\right) , $$ 
namely, $2S_q$ refines the divided differences of the sequences 
generated by $S_a$. Since the limits of the divided differences are values of the derivative of $S_a^\infty f^0$,
a convergent scheme $S_a$ generates $C^1$ limits if $2 S_q$ is convergent. We can check this sufficient condition by applying Algorithm~\ref{alg:conver_test_improved} to $2S_q =S_{2q}$ instead of $S_a$. For higher smoothness $C^{\, n}$, we look at the scheme refining the $n$-th order divided difference by considering the symbol decomposition when it exists,
\begin{equation} \label{eqn:smoothness}
    a(z) = (1+z)^{n} \, q(z) .
\end{equation}
For more details see \cite[Theorem 4]{dyn2002analysis}. Therefore, we can employ  Algorithm~\ref{alg:conver_test_improved} to check the convergence of $S_{2^n q}$, the scheme with symbol $2^n q(z)$ with $q(z)$ of~\eqref{eqn:smoothness}, which is a sufficient 
condition for the $C^{\, n}$ smoothness of $S_a$.

\begin{algorithm}[H]
\caption{Convergence of linear subdivision scheme: Improved version}
\label{alg:conver_test_improved}
\begin{algorithmic}[1]
\REQUIRE The symbol $a(z)$. The maximum number of iterations $M$.
\ENSURE The convergence status of the scheme
\IF{$a(1) \neq 2$ or $a(-1) \neq 0$} 
\RETURN \textit{``The scheme does not converge''}
\ENDIF
\STATE $q(z) \gets \frac{a(z)}{1+z}$.
\IF[Apply~Theorem~\ref{thm:diverge_by_q}]{$\abs{q(-1)} >1 $}
\RETURN \textit{``The scheme does not converge''}  
\ENDIF
\IF[Apply~Theorem~\ref{thm:positive_coefs}]{The coefficients of $q$ are non-negative}
\STATE  $\mu = \max\biggl\{ \sum_{j}q_{2j+1},\sum_{j}q_{2j} \biggr\}$
\RETURN \textit{``The scheme converges with contractivity factor $\mu$ and with contractivity number $1$''} 
\ENDIF
\STATE Set $q_{1}(z)=q(z)$ and denote $q_{1}(z)=\sum_{i} q_{i}^{[1]} z^{i}$.
\FOR{$L=1, \ldots, M$}
\IF{$\max_{0 \leq i<2^{L}} \sum_{j\in \Z}\abs{q_{i-2^{L} j}^{[L]}}<1$} 
\STATE  $\mu = \max_{0 \leq i<2^{L}} \sum_{j\in \Z}\abs{q_{i-2^{L} j}^{[L]}}$
\RETURN \textit{``The scheme is convergent with contractivity factor $\mu$ and with contractivity number $L$''}
\ELSE
\STATE  Set $q_{L+1}(z)=q(z) q_{L}\left(z^{2}\right)$ and denote $q_{L+1}(z)=\sum_{i} q_{i}^{[L+1]} z^{i}$.
\ENDIF
\ENDFOR     
\RETURN \textit{``Inconclusive! $\norm{S^k_{q}}\ge 1$ 
for all $k \le M$''}
\end{algorithmic}
\end{algorithm}



\bibliographystyle{plain} 
\bibliography{The_bib}

\end{document}